\documentclass[11pt]{amsart}
\usepackage[margin=30mm]{geometry}
\usepackage{amsmath,amssymb}
\usepackage{amsthm}
\usepackage{mathrsfs}

\newtheorem{thm}{Theorem}[section]
\newtheorem{prop}{Proposition}[section]
\newtheorem{lem}{Lemma}[section]
\newtheorem{cor}{Corollary}[section]

\newtheorem{ex}{Example}[section]
\newtheorem{rem}{Remark}[section]
\newtheorem*{lemma}{\it{Lemma}}

\usepackage{enumitem}

\begin{document}

\title{Shadowing and the continuity of omega-limit sets}
\author{Noriaki Kawaguchi}
\address{Department of Mathematical and Computing Science, School of Computing, Institute of Science Tokyo, 2-12-1 Ookayama, Meguro-ku, Tokyo 152-8552, Japan}
\email{gknoriaki@gmail.com}

\begin{abstract}
This paper examines the relationship between shadowing phenomena and the continuity properties of $\omega$-limit sets in dynamical systems. We give a necessary and sufficient condition for a shadowable point to be an upper (resp.\:a lower) semicontinuity point of $\omega$-limit sets. Assuming global shadowing, we show that the lower semicontinuity of $\omega$-limit sets is equivalent to the chain continuity. We also show that the lower semicontinuity of $\omega$-limit sets is equivalent to the chain continuity in a general setting. Several examples are given to illustrate the results.
\end{abstract}

\maketitle

\markboth{NORIAKI KAWAGUCHI}{Shadowing and the continuity of omega-limit sets}

\section{Introduction}

{\em Shadowing} is an important concept in the topological theory of dynamical systems, linking the behavior of coarse orbits, or {\em pseudo-orbits}, to that of true orbits. Since its introduction in the context of hyperbolic dynamics, various shadowing properties have been the subject of extensive research (see \cite{AH,P} for background). In \cite{M}, by localizing  global shadowing to individual points, Morales introduced the notion of {\em shadowable points}, which gives a method for a local study of shadowing phenomena. The shadowable points and their relationship to other dynamical properties are the subject of ongoing research.

The {\em $\omega$-limit set} is the set of points to which the orbit of each point accumulates. Its continuity as a set-valued map can be interpreted as a kind of stability of the orbits. In the previous studies, several results in \cite{A} concern global shadowing and the continuity of $\omega$-limit sets. In this paper, we study the (semi-)continuity of $\omega$-limit sets, relating it to shadowable points and also global shadowing. Our results improve those in \cite{A} considerably.

We begin by defining the (semi-)continuity of set-valued maps. Throughout, $X$ denotes a compact metric space endowed with a metric
\[
d\colon X\times X\to[0,\infty).
\]
We denote by $\mathcal{K}(X)$ the set of non-empty closed subsets of $X$. For a map $\phi\colon X\to\mathcal{K}(X)$ and $x\in X$, we say that $\phi$ is
\begin{itemize}
\item {\em upper semicontinuous} at $x$ if for every open subset $U$ of $X$ with $\phi(x)\subset U$, there is $\delta>0$ such that $d(x,y)\le\delta$ implies $\phi(y)\subset U$ for all $y\in X$,
\item {\em lower semicontinuous} at $x$ if for every open subset $U$ of $X$ with $\phi(x)\cap U\ne\emptyset$, there is $\delta>0$ such that $d(x,y)\le\delta$ implies $\phi(y)\cap U\ne\emptyset$ for all $y\in X$.
\end{itemize}
We denote by $USC(\phi)$ (resp.\:$LSC(\phi)$) the set of upper (resp.\:lower) semicontinuity points for $\phi$. For a non-empty subset $S$ of $X$ and $x\in X$, let $d(x,S)$ denote the distance of $x$ from $S$:
\[
d(x,S)=\inf_{y\in S}d(x,y).
\]
The {\em Hausdorff distance} $d_H\colon\mathcal{K}(X)\times\mathcal{K}(X)\to[0,\infty)$ is define by
\[
d_H(A,B)=\max\{\sup_{a\in A}d(a,B),\sup_{b\in B}d(b,A)\}
\]
for all $A,B\in\mathcal{K}(X)$. We know that $(\mathcal{K}(X),d_H)$ is a compact metric space. For a map $\phi\colon X\to\mathcal{K}(X)$, we denote by $C(\phi)$ the set of continuity points for $\phi$ (with respect to $d_H$). Note that
\[
C(\phi)=USC(\phi)\cap LSC(\phi).
\]

Given a continuous map $f\colon X\to X$ and $x\in X$, we denote
\begin{itemize}
\item by $\omega_f(x)$ the {\em $\omega$-limit set} of $x$ for $f$, i.e., the set of $y\in X$ such that $\lim_{j\to\infty}f^{i_j}(x)=y$ for some $0\le i_1<i_2<\cdots$,
\item by $\Omega_f(x)$ the set of $y\in X$ such that $\lim_{j\to\infty}x_j=x$ and $\lim_{j\to\infty}f^{i_j}(x_j)=y$ for some $x_j\in X$, $j\ge1$, and $0\le i_1<i_2<\cdots$.
\end{itemize}
For $\delta>0$, a sequence $(x_i)_{i=0}^k$ of points in $X$, where $k\ge1$, is said to be a {\em $\delta$-chain} of $f$ if
\[
\sup_{0\le i\le k-1}d(f(x_i),x_{i+1})\le\delta.
\]
We denote by $\omega_f^\ast(x)$ the set of $y\in X$ such that for any $\delta>0$ and $l\ge1$, there is a $\delta$-chain $(x_i)_{i=0}^k$ of $f$ with $x_0=x$, $x_k=y$, and $l\le k$. For $y,z\in X$, the notation $y\rightarrow z$ means that for every $\delta>0$, there is a  $\delta$-chain $(x_i)_{i=0}^k$ of $f$ with $x_0=y$ and $x_k=z$. We say that a closed subset $S$ of $X$ with $f(S)\subset S$ is {\em chain stable} if $y\rightarrow z$ implies $z\in S$ for all $y\in S$ and $z\in X$, or equivalently, for any $\epsilon>0$, there is $\delta>0$ such that every $\delta$-chain $(x_i)_{i=0}^k$ of $f$ with $x_0\in S$ satisfies $d(x_k,S)\le\epsilon$. Note that  
\begin{itemize}
\item $S$ is a closed subset of $X$ and satisfies $f(S)=S$ for all $S\in\{\omega_f(x),\Omega_f(x),\omega_f^\ast(x)\}$,
\item $\omega_f(x)\subset\Omega_f(x)\subset\omega_f^\ast(x)$,
\item $\lim_{i\to\infty}d(f^i(x),\omega_f(x))=0$,
\item $x\rightarrow y$ for all $y\in\omega_f(x)$,
\item $y\rightarrow z$ for all $y\in\omega_f(x)$ and $z\in\omega_f^\ast(x)$; therefore,
\[
\omega_f^\ast(x)=\{z\in X\colon y\rightarrow z\:\:\text{for some $y\in\omega_f(x)$}\},
\]
\item $\omega_f^\ast(x)$ is chain stable.
\end{itemize}

\begin{rem}
\normalfont
By \cite[Proposition 7.22]{A1}, we know that for any continuous map $f\colon X\to X$,
\[
\{x\in X\colon\omega_f(x)=\Omega_f(x)\}
\]
is a dense $G_\delta$-subset of $X$.
\end{rem}

The definition of shadowable points is as follows. Let $f\colon X\to X$ be a continuous map and let $\xi=(x_i)_{i\ge0}$ be a sequence of points in $X$. For $\delta>0$, $\xi$ is said to be a {\em $\delta$-pseudo orbit} of $f$ if
\[
\sup_{i\ge0}d(f(x_i),x_{i+1})\le\delta.
\]
For $y\in X$ and $\epsilon>0$, $\xi$ is said to be {\em $\epsilon$-shadowed} by $y$ if
\[
\sup_{i\ge0}d(x_i,f^i(y))\le\epsilon.
\]
Given $x\in X$, $x$ is said to be a {\em shadowable point} for $f$ if for any $\epsilon>0$, there is $\delta>0$ such that every $\delta$-pseudo orbit $(x_i)_{i\ge0}$ of $f$ with $x_0=x$ is $\epsilon$-shadowed by some $y\in X$. We denote by $Sh(f)$ the set of shadowable points for $f$.

\begin{rem}
For any continuous map $f\colon X\to X$ and $x\in X$, if $x\in Sh(f)$, then $\Omega_f(x)=\omega_f^\ast(x)$.
\end{rem}

This paper consists of seven sections and an appendix.

\begin{itemize}
\item In Section 2 (resp.\:3), we give a necessary and sufficient condition for a shadowable point to be an upper (resp.\:a lower) semicontinuity point of $\omega$-limit sets.
\item In Section 4, we present several corollaries of the results in the previous sections, introducing the notions of chain recurrence, chain components, and chain continuity. We also introduce some notations, definitions, and facts that are used in subsequent sections.
\item In Section 5, we show that under the assumption of global shadowing, the lower semicontinuity of $\omega$-limit sets is equivalent to the chain continuity. We also prove a general result about chain continuity.
\item In Section 6, we show that the lower semicontinuity of $\omega$-limit sets is equivalent to the chain continuity in a general setting.
\item In Section 7, we present several examples to illustrate the results.
\item In Appendix A, we show that the upper semicontinuity of $\omega$-limit sets implies the lower semicontinuity of them.
\end{itemize}

\section{Upper semicontinuity}

The aim of this section is to prove the following theorem. 

\begin{thm}
Given a continuous map $f\colon X\to X$ and $x\in Sh(f)$, the following conditions are equivalent
\begin{itemize}
\item $x\in USC(\omega_f)$,
\item $\omega_f(x)=\Omega_f(x)$.
\end{itemize}
\end{thm}

Let us first prove the following lemma.

\begin{lem}
For any continuous map $f\colon X\to X$,
\[
\{x\in X\colon\omega_f(x)=\Omega_f(x)\}\subset USC(\omega_f).
\] 
\end{lem}

\begin{proof}
Let $x\in X$ and $\omega_f(x)=\Omega_f(x)$. If $x\not\in USC(\omega_f)$, then there are an open subset $U$ of $X$, $x_j,y_j\in X$, $j\ge1$, and $y\in X$ such that
\begin{itemize}
\item $\omega_f(x)\subset U$,
\item $\lim_{j\to\infty}x_j=x$,
\item $y_j\in\omega_f(x_j)\setminus U$ for all $j\ge1$,
\item $\lim_{j\to\infty}y_j=y$.
\end{itemize}
It follows that
\[
y\in\Omega_f(f)\setminus U\subset\Omega_f(f)\setminus\omega_f(x),
\]
a contradiction; therefore, the lemma has been proved.
\end{proof}

Theorem 2.1 is a direct consequence of Lemma 2.1 and the next lemma.

\begin{lem}
For any continuous map $f\colon X\to X$,
\[
Sh(f)\cap USC(\omega_f)\subset\{x\in X\colon\omega_f(x)=\Omega_f(x)\}.
\] 
\end{lem}

\begin{proof}
Let $x\in Sh(f)\cap USC(\omega_f)$ and assume $\omega_f(x)\ne\Omega_f(x)$. We fix $y\in\Omega_f(x)\setminus\omega_f(x)$ and $z\in\omega_f(x)$. Since $z\in\omega_f(x)$ and $y\in\Omega_f(x)$, we have $x\rightarrow z$ and $z\rightarrow y$. By $x\in USC(\omega_f)$ and $y\in\Omega_f(x)$, we obtain $y\rightarrow w$ for some $w\in\omega_f(x)$. Since $z,w\in\omega_f(x)$, we have $w\rightarrow z$. From $y\rightarrow w$ and $w\rightarrow z$, it follows that $y\rightarrow z$. For any $\epsilon>0$, by $x\in Sh(f)$, there is $\delta>0$ such that every $\delta$-pseudo orbit $(w_i)_{i\ge0}$ of $f$ with $w_0=x$ is $\epsilon$-shadowed by some $p\in X$. We take $\delta$-chains $\alpha=(x_i)_{i=0}^k$, $\beta=(y_i)_{i=0}^l$, $\gamma=(z_i)_{i=0}^m$ of $f$ such that $x_0=x$, $x_k=y_0=z_m=z$, and $y_l=z_0=y$. Consider the following $\delta$-pseudo orbit of $f$
\[
\xi=(w_i)_{i\ge0}=\alpha\beta\gamma\beta\gamma\beta\gamma\cdots
\]
with $w_0=x$, which is $\epsilon$-shadowed by some $p_\epsilon\in X$. Note that $d(x,p_\epsilon)\le\epsilon$. Since
\[
d(y,f^{k+l+n(l+m)}(p_\epsilon))\le\epsilon 
\]
for all $n\ge0$, we obtain $d(y,q_\epsilon)\le\epsilon$ for some $q_\epsilon\in\omega(p_\epsilon)$. Note that $y\not\in\omega_f(x)$. As $\epsilon>0$ is arbitrary, we obtain $x\not\in USC(\omega_f(x))$, a contradiction. This completes the proof of the lemma. 
\end{proof}

The following is a corollary of Theorem 2.1.

\begin{cor}
Given a continuous map $f\colon X\to X$ and $x\in Sh(f)$, the following conditions are equivalent
\begin{itemize}
\item $x\in USC(\omega_f)$,
\item $\omega_f(x)$ is chain stable.
\end{itemize}
\end{cor}

\section{Lower semicontinuity}

We shall recall the definition of minimal sets. Given a continuous map $f\colon X\to X$, a non-empty closed subset $M$ of $X$ with $f(M)\subset M$ is called a {\em minimal set} for $f$ if $M=\omega_f(x)$ for all $x\in M$, or equivalently, for any closed subset $S$ of $X$ with $S\subset M$, $f(S)\subset S$ implies $S\in\{\emptyset,M\}$. By Zorn's lemma, for every non-empty closed subset $S$ of $X$ with $f(S)\subset S$, we have a minimal set $M$ for $f$ such that $M\subset S$.

The aim of this section is to prove the following theorem. 

\begin{thm}
Given a continuous map $f\colon X\to X$ and $x\in Sh(f)$, the following conditions are equivalent
\begin{itemize}
\item $x\in LSC(\omega_f)$,
\item $\omega_f(x)\subset\omega_f(y)$ for all $y\in\Omega_f(x)$ (and so $\omega_f(x)$ is a minimal set for $f$).
\end{itemize}
\end{thm}

Let us first prove the following lemma.

\begin{lem}
For any continuous map $f\colon X\to X$,
\[
\{x\in X\colon\text{$\omega_f(x)\subset\omega_f(y)$ for all $y\in\Omega_f(x)$}\}\subset LSC(\omega_f).
\]
\end{lem}

\begin{proof}
Let $x\in X$ and let $\omega_f(x)\subset\omega_f(y)$ for all $y\in\Omega_f(x)$. If $x\not\in LSC(\omega_f)$, then there are an open subset $U$ of $X$, $x_j\in X$, $j\ge1$, $x\in X$, and a closed subset $A$ of $X$ such that
\begin{itemize}
\item $\omega_f(x)\cap U\ne\emptyset$,
\item $\lim_{j\to\infty}x_j=x$,
\item $\omega_f(x_j)\cap U=\emptyset$ for all $j\ge1$,
\item $\lim_{j\to\infty}\omega_f(x_j)=A$ (with respect to the Hausdorff distance).
\end{itemize}
It follows that
\begin{itemize}
\item $A\cap U=\emptyset$ and so $\omega_f(x)\not\subset A$,
\item $A\subset\Omega_f(x)$,
\item $f(A)\subset A$.
\end{itemize}
By taking $y\in A$, as $\omega_f(y)\subset A$, we obtain $y\in \Omega_f(x)$ and $\omega_f(x)\not\subset\omega_f(y)$, a contradiction. Thus, we obtain $x\in LSC(\omega_f)$, completing the proof.
\end{proof}

Theorem 3.1 is an immediate consequence of Lemma 3.1 and the next lemma.

\begin{lem}
For any continuous map $f\colon X\to X$,
\[
Sh(f)\cap LSC(\omega_f)\subset\{x\in X\colon\text{$\omega_f(x)\subset\omega_f(y)$ for all $y\in\Omega_f(x)$}\}.
\]
\end{lem}

For $x\in X$ and $r>0$, we denote by $B_r(x)$ the closed $r$-ball centered at $x$:
\[
B_r(x)=\{y\in X\colon d(x,y)\le r\}.
\]
Similarly, for a subset $S$ of $X$ and $r>0$, let $B_r(S)$ denote the closed $r$-neighborhood of $S$:
\[
B_r(S)=\{y\in X\colon d(y,S)\le r\}.
\]

\begin{proof}
Let $x\in Sh(f)\cap LSC(\omega_f)$ and assume that $\omega_f(x)\not\subset\omega_f(y)$ for some $y\in\Omega_f(x)$. We take $z\in\omega_f(x)\setminus\omega_f(y)$ and $0<r<d(z,\omega_f(y))/2$. Note that $B_r(z)\cap B_r(\omega_f(y))=\emptyset$. For any $0<\epsilon\le r$, by $x\in Sh(f)$, there is $\delta>0$ such that every $\delta$-pseudo orbit $(x_i)_{i\ge0}$ of $f$ with $x_0=x$ is $\epsilon$-shadowed by some $p\in X$. Since $y\in\Omega_f(x)$, we have $x\rightarrow y$ and so there is a $\delta$-chain $(y_i)_{i=0}^k$ of $f$ with $y_0=x$ and $y_k=y$. Consider the following $\delta$-pseudo orbit of $f$
\[
\xi=(x_i)_{i\ge0}=(y_0,y_1,\dots,y_{k-1},y,f(y),f^2(y),\dots)
\]
with $x_0=x$, which is $\epsilon$-shadowed by some $p_\epsilon\in X$. Note that $d(x,p_\epsilon)\le\epsilon$. By
\[
\omega_f(p_\epsilon)\subset B_\epsilon(\omega_f(y))\subset B_r(\omega_f(y))
\]
we have $B_r(z)\cap\omega_f(p_\epsilon)=\emptyset$. Since $0<\epsilon\le r$ is arbitrary, we obtain $x\not\in LSC(\omega_f(x))$, a contradiction. This completes the proof of the lemma.
\end{proof}

The following is a corollary of Lemma 3.2.

\begin{cor}
Given a continuous map $f\colon X\to X$ and $x\in Sh(f)$, if $x\in LSC(\omega_f)$, then $y\rightarrow z$ for all $y,z\in\omega_f^\ast(x)$.
\end{cor}

\begin{proof}
Let $x\in Sh(f)\cap LSC(\omega_f)$ and note that $\Omega_f(x)=\omega_f^\ast(x)$. By Lemma 3.2, we have $\omega_f(x)\subset\omega_f(w)$ for all $w\in\omega_f^\ast(x)$. Let $y,z\in\omega_f^\ast(x)$. Since $f(\omega_f^\ast(x))=\omega_f^\ast(x)$, there is a sequence $x_i\in\omega_f^\ast(x)$, $i\le0$, such that
\begin{itemize}
\item $x_0=z$,
\item $f(x_{i-1})=x_i$ for all $i\le0$. 
\end{itemize}
Let $\alpha$ be the set of $p\in X$ such that $\lim_{j\to\infty}x_{i_j}=p$ for some $0\ge i_1>i_2>\cdots$. Note that
\begin{itemize}
\item $\alpha$ is a closed subset of $\omega_f^\ast(x)$ and satisfies $f(\alpha)=\alpha$,
\item $\lim_{i\to-\infty}d(x_i,\alpha)=0$,
\item $p\rightarrow q$ for all $p,q\in\alpha$.
\end{itemize}
By taking $q\in\alpha$, we have $q\rightarrow z$. Since
\[
\omega_f(x)\subset\omega_f(y)\cap\omega_f(q)\subset\omega_f(y)\cap\alpha,
\]
by taking $p\in\omega_f(x)$, we obtain $y\rightarrow p$ and $p\rightarrow q$. It follows that $y\rightarrow z$. Since $y,z\in\omega_f^\ast(x)$ are arbitrary, the corollary has been proved.
\end{proof}

\section{Corollaries}

The main aim of this section is to present several corollaries of the results in the previous sections. We shall recall the notion of chain recurrence and chain components. Given a continuous map $f\colon X\to X$, the {\em chain recurrent set} $CR(f)$ for $f$ is defined as
\[
CR(f)=\{x\in X\colon x\to x\}.
\]
Note that $CR(f)$ is a closed subset of $X$ and satisfies $f(CR(f))=CR(f)$. An equivalence class of the (equivalence) relation $\leftrightarrow$ in
\[
CR(f)^2=CR(f)\times CR(f)
\]
defined by: for all $x,y\in CR(f)$, $x\leftrightarrow y$ if and only if $x\rightarrow y$ and $y\rightarrow x$, is called a {\em chain component} for $f$. Let $\mathcal{C}(f)$ denote the set of chain components for $f$. Every $C\in\mathcal{C}(f)$ is a closed subset of $CR(f)$ and satisfies $f(C)=C$. For any $x\in X$, since $y\rightarrow z$ for all $y,z\in\omega_f(x)$, there is a unique $C_f(x)\in\mathcal{C}(f)$ such that $\omega_f(x)\subset C_f(x)$. Note that
\[
\omega_f(x)\subset C_f(x)\subset\omega_f^\ast(x).
\]
Following \cite{AHK}, we say that a chain component $C\in\mathcal{C}(f)$ is {\em terminal} if $C$ is chain stable.  We denote by $\mathcal{C}_{\rm ter}(f)$ the set of terminal chain components for $f$.

\begin{rem}
\normalfont
Let $f\colon X\to X$ be a continuous map. For any connected subset $S$ of $X$, if $S\subset CR(f)$, then we have $x\rightarrow y$ for all $x,y\in S$.
\end{rem}

The following is a consequence of Corollaries 2.1 and 3.1. 

\begin{cor}
Given a continuous map $f\colon X\to X$ and $x\in Sh(f)$, if $x\in USC(\omega_f)\cup LSC(\omega_f)$, then $C_f(x)\in\mathcal{C}_{\rm ter}(f)$.
\end{cor}

By Theorems 2.1 and 3.1, we obtain the following corollary.

\begin{cor}
Given a continuous map $f\colon X\to X$ and $x\in Sh(f)$, the following conditions are equivalent
\begin{itemize}
\item[(1)] $x\in C(\omega_f)$,
\item[(2)] $\omega_f(x)$ is chain stable and a minimal set for $f$,
\item[(3)] $C_f(x)\in\mathcal{C}_{\rm ter}(f)$ and $C_f(x)$ is a minimal set for $f$.
\end{itemize}
\end{cor}

The notion of {\em chain continuity} was introduced in \cite{A2}. Let $f\colon X\to X$ be a continuous map and let $x\in X$. We say that $f$ is
\begin{itemize}
\item {\em equicontinuous} at $x$ if for any $\epsilon>0$, there is $\delta>0$ such that $d(x,y)\le\delta$ implies
\[
\sup_{i\ge0}d(f^i(x),f^i(y))\le\epsilon
\]
for all $y\in X$,
\item {\em chain continuous} at $x$ if for any $\epsilon>0$, there is $\delta>0$ such that every $\delta$-pseudo orbit $(x_i)_{i\ge0}$ of $f$ with $x_0=x$ is $\epsilon$-shadowed by $x$, i.e., satisfies
\[
\sup_{i\ge0}d(x_i,f^i(x))\le\epsilon.
\]
\end{itemize}
We denote by $EC(f)$ (resp.\:$CC(f)$) the set of equicontinuity (resp.\:chain continuity) points for $f$. Note that
\[
CC(f)=Sh(f)\cap EC(f) 
\]
and $EC(f)\subset C(\omega_f)$.

Given a continuous map $f\colon X\to X$, let $h_{\rm top}(f)$ denote the topological entropy of $f$ and let $Ent(f)$ denote the set of entropy points for $f$. Note that $h_{\rm top}(f)>0$ if and only if $Ent(f)\ne\emptyset$ (see \cite{YZ}). 

\begin{rem}
\normalfont
Let $f\colon X\to X$ be a continuous map and let $x\in X$.
\begin{itemize}
\item It is known that $x\in CC(f)$ if and only if
\begin{itemize}
\item $C_f(x)\in\mathcal{C}_{\rm ter}(f)$,
\item $C_f(x)$ is a periodic orbit or an odometer
\end{itemize}
(see, e.g., \cite[Theorem 7.5]{AHK}).
\item For any $C\in\mathcal{C}_{\rm ter}(f)$, if $X$ is locally connected and $C$ is totally disconnected, then $C$ is a periodic orbit or an odometer (see, e.g., \cite{HH}).
\item We know that if $x\in Sh(f)\setminus Ent(f)$, then every $C\in\mathcal{C}(f)$ with $x\rightarrow y$ for some $y\in C$ is a periodic orbit or an odometer (see \cite[Theorem 1.1]{K}). 
\end{itemize}
\end{rem}

By Remark 4.2, we have the following lemma.

\begin{lem}
Given a continuous map $f\colon X\to X$ and $x\in X$, let $C_f(x)\in\mathcal{C}_{\rm ter}(f)$.
\begin{itemize}
\item[(1)] If $X$ is locally connected and $C_f(x)$ is totally disconnected, then $x\in CC(f)$.
\item[(2)] If $x\in Sh(f)\setminus Ent(f)$, then $x\in CC(f)$ and so $x\in C(\omega_f)$.
\end{itemize}
\end{lem}

By Corollary 4.1 and Lemma 4.1, we obtain the following corollaries.

\begin{cor}
Let $f\colon X\to X$ be a continuous map and let $x\in Sh(f)$.
\begin{itemize}
\item If $C_f(x)\in\mathcal{C}_{\rm ter}(f)$ and $x\not\in C(\omega_f)$, then $x\in Ent(f)$.
\item If $x\in USC(\omega_f)\cup LSC(\omega_f)$ and $x\not\in CC(f)$, then $x\in Ent(f)$.
\item If $x\in USC(\omega_f)\setminus LSC(\omega_f)$ or $x\in LSC(\omega_f)\setminus USC(\omega_f)$, then $x\in Ent(f)$.
\end{itemize}
\end{cor}

Given a continuous map $f\colon X\to X$ and $x\in X$, $x$ is called a {\em regularly recurrent point} for $f$ if for any $\epsilon>0$, there is $k\ge1$ such that
\[
\sup_{j\ge0}d(x,f^{jk}(x))\le\epsilon.
\]
We denote by $RR(f)$ the set of regularly recurrent points for $f$. For use in the next section, we shall prove the following.

\begin{lem}
Every continuous map $f\colon X\to X$ satisfies $CR(f)\cap CC(f)\subset RR(f)$.
\end{lem}

\begin{proof}
Let $x\in CR(f)\cap CC(f)$. For any $\epsilon>0$, since $x\in CC(f)$, there is $\delta>0$ such that every $\delta$-pseudo orbit $(x_i)_{i\ge0}$ of $f$ with $x_0=x$ satisfies
\[
\sup_{i\ge0}d(x_i,f^i(x))\le\epsilon.
\]
Since $x\in CR(f)$, we have a $\delta$-chain $\alpha=(y_i)_{i=0}^k$ of $f$ with $y_0=y_k=x$. Consider the periodic $\delta$-pseudo orbit
\[
\xi=(x_i)_{i\ge0}=\alpha\alpha\alpha\cdots
\]
of $f$. It follows that
\[
\sup_{j\ge0}d(x,f^{jk}(x))=\sup_{j\ge0}d(x_{jk},f^{jk}(x))\le\epsilon.
\]
Since $\epsilon>0$ is arbitrary, we obtain $x\in RR(f)$; therefore, $CR(f)\cap CC(f)\subset RR(f)$, proving the lemma. 
\end{proof}

\section{Chain continuity (I)}

The aim of this section is to prove the following two statements.

\begin{thm}
Given a continuous map $f\colon X\to X$ with $X=Sh(f)$, the following conditions are equivalent
\begin{itemize}
\item $X=LSC(\omega_f)$,
\item $X=CC(f)$.
\end{itemize}
\end{thm}

\begin{prop}
For a continuous map $f\colon X\to X$ such that $X$ is connected, the following conditions are equivalent
\begin{itemize}
\item $X=CC(f)$,
\item $\bigcap_{i\ge0}f^i(X)$ is a singleton.
\end{itemize}
\end{prop}

\begin{rem}
\normalfont
In Appendix A, we show that for any continuous map $f\colon X\to X$, $X=USC(\omega_f)$ implies $X=LSC(\omega_f)$ and so $X=C(\omega_f)$.
\end{rem}

Given a continuous map $f\colon X\to X$ and $x\in X$, $x$ is called a {\em minimal point} for $f$ if
\begin{itemize}
\item $x\in\omega_f(x)$,
\item $\omega_f(x)$ is a minimal set for $f$.
\end{itemize}
We denote by $M(f)$ the set of minimal points for $f$. Note that
\begin{itemize}
\item $M(f)$ is the (disjoint) union of all minimal sets for $f$,
\item $M(f)=M(f^m)$ for all $m\ge1$ (see, e.g., \cite[Lemma V.7]{BC}).
\end{itemize}

Let us first prove the following lemma.

\begin{lem}
For a continuous map $f\colon X\to X$, if
\begin{itemize}
\item $X=Sh(f)$,
\item $\omega_f(x)\subset M(f)$ for all $x\in X$,
\item $C_f(x)\in\mathcal{C}_{\rm ter}(f)$ for all $x\in X$,
\end{itemize}
then $X=CC(f)$.
\end{lem}

\begin{proof}
Since $X=Sh(f)$, if $h_{\rm top}(f)>0$, then there are $m\ge1$ and a closed subset $Y$ of $X$ with the following properties:
\begin{itemize}
\item $f^m(Y)\subset Y$,
\item there is a continuous surjection $\pi\colon Y\to\{0,1\}^\mathbb{N}$ such that $\pi\circ f^m|_{Y}=\sigma\circ\pi $, where $\sigma\colon\{0,1\}^\mathbb{N}\to\{0,1\}^\mathbb{N}$ is the shift map
\end{itemize}
(see, e.g., \cite[Theorem 3.7]{LO}). We take $p\in\{0,1\}^\mathbb{N}$ with $\omega_\sigma(p)=\{0,1\}^\mathbb{N}$. By taking $q\in Y$ with $\pi(q)=p$, we obtain
\[
\pi(\omega_{f^m}(q))=\omega_\sigma(\pi(q))=\omega_\sigma(p)=\{0,1\}^\mathbb{N}
\]
and so $p=\pi(r)$ for some $r\in\omega_{f^m}(q)$. By
\[
r\in\omega_{f^m}(q)\subset\omega_f(q)\subset M(f)=M(f^m),
\]
we obtain $p\in M(\sigma)$; therefore, $\{0,1\}^\mathbb{N}$ is a minimal set for $\sigma$, a contradiction. It follows that $h_{\rm top}(f)=0$ and thus $Ent(f)=\emptyset$. For every $x\in X$, by $C_f(x)\in\mathcal{C}_{\rm ter}(f)$ and $x\in Sh(f)\setminus Ent(f)$, by Lemma 4.1, we obtain $x\in CC(f)$. We conclude that $X=CC(f)$, completing the proof.
\end{proof}

By using Lemma 5.1, we prove Theorem 5.1.

\begin{proof}[Proof of Theorem 5.1]
Assume that $X=LSC(\omega_f)$. Since
\[
X=Sh(f)\cap LSC(\omega_f),
\]
by Corollary 4.1, we have $C_f(x)\in\mathcal{C}_{\rm ter}(f)$ for all $x\in X$. For any $x\in X$, as
\[
x\in Sh(f)\cap LSC(\omega_f),
\]
Theorem 3.1 implies that $\omega_f(x)$ is a minimal set for $f$. It follows that $\omega_f(x)\subset M(f)$ for all $x\in X$. By Lemma 5.1, we obtain $X=CC(f)$, completing the proof of the theorem.
\end{proof}

The proof of Proposition 5.1 relies on the following two lemmas.

\begin{lem}
If a continuous map $f\colon X\to X$ is surjective and satisfies $X=CC(f)$, then $X=CR(f)$.
\end{lem}

\begin{proof}
Let $x\in X$. Since $f$ is surjective, there is a sequence $x_i\in X$, $i\le0$, such that
\begin{itemize}
\item $x_0=x$,
\item $f(x_{i-1})=x_i$ for all $i\le0$. 
\end{itemize}
Let $\alpha$ be the set of $y\in X$ such that $\lim_{j\to\infty}x_{i_j}=y$ for some $0\ge i_1>i_2>\cdots$. Note that
\begin{itemize}
\item $\alpha$ is a closed subset of $X$ and satisfies $f(\alpha)=\alpha$,
\item $\lim_{i\to-\infty}d(x_i,\alpha)=0$,
\item $y\rightarrow z$ for all $y,z\in\alpha$.
\end{itemize}
By taking $y\in\alpha$, we obtain $y\rightarrow x$. This with $y\in CC(f)$ implies $x\in\alpha$. It follows that $x\rightarrow x$, i.e., $x\in CR(f)$. Since $x\in X$ is arbitrary, we conclude that $X=CR(f)$, proving the lemma.
\end{proof}

\begin{lem}
If a continuous map $f\colon X\to X$ is surjective and satisfies $X=CC(f)$, then $X$ is totally disconnected.
\end{lem}

\begin{proof}
Assume the contrary, i.e., there is a non-empty connected subset $S$ of $X$ which is not a singleton. We take $p,q,r\in S$ with $q\ne r$. Let $g=f\times f\colon X\times X\to X\times X$. Since $g$ is surjective and satisfies
\[
CC(g)=CC(f)\times CC(f)=X\times X,
\]
by Lemma 5.2, we have $X\times X=CR(g)$. Since $S\times S$ is a connected subset of $X\times X$, we have $(x,y)\rightarrow(z,w)$ for all $(x,y),(z,w)\in S\times S$; therefore, $(p,p)\rightarrow(q,r)$. It follows that for every $\delta>0$, there is a pair
\[
((x_i)_{i=0}^k,(y_i)_{i=0}^k)
\]
of $\delta$-chains of $f$ such that $x_0=y_0=p$ and $(x_k,y_k)=(q,r)$. Since $q\ne r$, we obtain $p\not\in CC(f)$, a contradiction. This completes the proof of the lemma. 
\end{proof}

Let us complete the proof of Proposition 5.1.

\begin{proof}[Proof of Proposition 5.1]
Let $S=\bigcap_{i\ge0}f^i(X)$ and note that $S$ is connected. Let $g=f|_S\colon S\to S$. If $X=CC(f)$, then $g$ is surjective and satisfies $S=CC(g)$. From Lemma 5.3, it follows that $S$ is totally disconnected and thus $S$ is a singleton. Conversely, if $S$ is a singleton, then letting $S=\{p\}$, we have $CR(f)=\{p\}$. For any $x\in X$, by $C_f(x)=\{p\}\in\mathcal{C}_{\rm ter}(f)$, we obtain $x\in CC(f)$. Since $x\in X$ is arbitrary, we conclude that $X=CC(f)$, completing the proof.
\end{proof}

Given a continuous map $f\colon X\to X$, we denote by $R(f)$ the set of {\em recurrent points} for $f$:
\[
R(f)=\{x\in X\colon x\in\omega_f(x)\}.
\]
Let us prove the following lemma.

\begin{lemma}
For a continuous map $f\colon X\to X$, if $X=R(f)$, then $X=LSC(\omega_f)$.
\end{lemma}

\begin{proof}
Given any $x\in X$ and any open subset $U$ of $X$ with $\omega_f(x)\cap U\ne\emptyset$, we have $f^i(x)\in U$ for some $i\ge0$. As $U$ is open in $X$, there is $\delta>0$ such that $d(x,y)\le\delta$ implies $f^i(y)\in U$ for all $y\in X$. Since $X=R(f)$, for each $y\in X$, we have $y\in\omega_f(y)$ and so $f^i(y)\in\omega_f(y)$. It follows that for any $y\in X$, $d(x,y)\le\delta$ implies $f^i(y)\in\omega_f(y)\cap U$; therefore, $\omega_f(y)\cap U\ne\emptyset$. Since $x\in X$ and $U$ with $\omega_f(x)\cap U\ne\emptyset$ are arbitrary, we obtain $X=LSC(\omega_f)$, completing the proof.
\end{proof}

We end this section with some remarks.

\begin{rem}
\normalfont
\begin{itemize}
\item In \cite[Theorem 3.2]{A}, it is shown that for any continuous map $f\colon X\to X$ with $X=Sh(f)$, if
\begin{itemize}
\item $X$ is connected,
\item $f$ is surjective and satisfies $X=C(\omega_f)$,
\end{itemize}
then $X$ is a minimal set for $f$. According to Theorem 5.1 and Proposition 5.1, $X$ is actually a singleton in such a case.
\item In \cite[Proposition 3.4]{A}, it is shown that for any continuous map $f\colon X\to X$ with $X=Sh(f)$, $X=C(\omega_f)$ implies $CR(f)=M(f)$. This result can be improved as follows: by Theorem 5.1 and Lemma 4.2, for any continuous map $f\colon X\to X$ with $X=Sh(f)$, $X=LSC(\omega_f)$ implies $X=CC(f)$ and thus $CR(f)=RR(f)$.
\item  Given a continuous map $f\colon X\to X$, let $\mathcal{M}_f$ denote the set of minimal sets for $f$. In \cite[Proposition 3.5]{A}, it is shown that if $X=Sh(f)$ and $X=R(f)$, then $\mathcal{M}_f$ is a closed subset of $\mathcal{K}(X)$. This can be also proved as follows. Since $X=R(f)$, by the above lemma, we have $X=LSC(\omega_f)$. Since $X=Sh(f)$ and $X=LSC(\omega_f)$, Theorem 5.1 implies that $X=CC(f)$ and so $X=C(\omega_f)$. As $X=R(f)$ implies $X=CR(f)$, by Lemma 4.2, we obtain $X=RR(f)$ and so $X=M(f)$. It follows that $\mathcal{M}_f=\omega_f(X)$; therefore, $\mathcal{M}_f$ is a closed subset of $\mathcal{K}(X)$.
\end{itemize}
\end{rem}

\section{Chain continuity (I\!I)}

The aim of this section is to prove the following theorem.

\begin{thm}
For a continuous map $f\colon X\to X$ such that
\begin{itemize}
\item $X$ is connected,
\item $CR(f)$ is totally disconnected,
\end{itemize}
the following conditions are equivalent
\begin{itemize}
\item $X=LSC(\omega_f)$,
\item $X=CC(f)$.
\end{itemize}
\end{thm}

\begin{rem}
\normalfont
The chain recurrent set $CR(f)$ for $f$ in the assumption of Theorem 6.1 cannot be changed to the non-wandering set $\Omega(f)$ for $f$ (see Example 7.7).
\end{rem}

The proof of Theorem 6.1 is through the following lemma.

\begin{lem}
Let $f\colon X\to X$ be a continuous map and let $S=\bigcap_{i\ge0}f^i(X)$. If $X=LSC(\omega_f)$, then $S=CR(f)$.
\end{lem}

\begin{proof}
Given any $x\in S$, there is a sequence $x_i\in X$, $i\le0$, such that
\begin{itemize}
\item $x_0=x$,
\item $f(x_{i-1})=x_i$ for all $i\le0$. 
\end{itemize}
Let $\alpha$ be the set of $y\in X$ such that $\lim_{j\to\infty}x_{i_j}=y$ for some $0\ge i_1>i_2>\cdots$. Note that
\begin{itemize}
\item $\alpha$ is a closed subset of $X$ and satisfies $f(\alpha)=\alpha$,
\item $\lim_{i\to-\infty}d(x_i,\alpha)=0$.
\end{itemize}
We fix $y\in\alpha$ and a sequence $0\ge i_1>i_2>\cdots$ with $\lim_{j\to\infty}x_{i_j}=y$. Note that
\[
\omega_f(x)=\omega_f(x_{i_j})
\]
for all $j\ge1$. As $\omega_f(y)\subset\alpha$, by taking $z\in\omega_f(y)$, we obtain $z\in\alpha$. Since $y\in LSC(\omega_f)$ and $z\in\omega_f(y)$, for any open subset $U$ of $X$ with $z\in U$, we have
\[
\omega_f(x)\cap U=\omega_f(x_{i_j})\cap U\ne\emptyset
\]
for some $j\ge1$. It follows that $z\in\omega_f(x)$. By $z\in\omega_f(x)\cap\alpha$, we obtain $x\rightarrow z$ and $z\rightarrow x$, which implies $x\rightarrow x$, i.e., $x\in CR(f)$. Since $x\in S$ is arbitrary, we obtain $S\subset CR(f)$; therefore, $S=CR(f)$, completing the proof.
\end{proof}

By Lemma 6.1, we prove Theorem 6.1.

\begin{proof}[Proof of Theorem 6.1]
Let $S=\bigcap_{i\ge0}f^i(X)$ and note that $S$ is connected. If $X=LSC(\omega_f)$, then by Lemma 6.1, we have $S=CR(f)$. As $CR(f)$ is totally disconnected and so is $S$, $S$ is a singleton. This implies that $X=CC(f)$, thus the theorem has been proved.
\end{proof}

\section{Examples}

In this section, we present several examples to illustrate the results in this paper.

\begin{ex}
\normalfont
Let $\sigma\colon\{0,1\}^\mathbb{Z}\to\{0,1\}^\mathbb{Z}$ be the shift map. We define $x=(x_i)_{i\in\mathbb{Z}},y=(y_i)_{i\in\mathbb{Z}}$ by
\begin{itemize}
\item $x_i=0$ for all $i\in\mathbb{Z}$,
\item $y_0=1$ and $y_i=0$ for all $i\in\mathbb{Z}\setminus\{0\}$.
\end{itemize}
Let $X=\{x\}\cup\{\sigma^i(y)\colon i\in\mathbb{Z}\}$ and let $f=\sigma|_X\colon X\to X$. It holds that
\begin{itemize}
\item $X=C(\omega_f)$ because $\omega_f(z)=\{x\}$ for all $z\in X$,
\item $\omega_f(x)=\{x\}$ and $\Omega_f(x)=X$,
\item $Sh(f)=\emptyset$.
\end{itemize}
\end{ex}

\begin{ex}
\normalfont
Let $X=[0,1]$ and let $f\colon X\to X$ be the map defined as $f(x)=x^2$ for all $x\in X$. It holds that
\begin{itemize}
\item $X=Sh(f)$,
\item $USC(\omega_f)=LSC(\omega_f)=X\setminus\{1\}$,
\item $\omega_f(1)=\{1\}$ and $\Omega_f(1)=X$,
\item $\omega_f(x)=\Omega_f(x)=\{0\}$ for all $x\in X\setminus\{1\}$.
\end{itemize}
\end{ex}

\begin{ex}
\normalfont
Let $\sigma\colon\{0,1\}^\mathbb{Z}\to\{0,1\}^\mathbb{Z}$ be the shift map. We define $x=(x_i)_{i\in\mathbb{Z}},y=(y_i)_{i\in\mathbb{Z}},z=(z_i)_{i\in\mathbb{Z}}$ by
\begin{itemize}
\item $x_i=0$ for all $i\in\mathbb{Z}$,
\item $y_0=1$ and $y_i=0$ for all $i\in\mathbb{Z}\setminus\{0\}$,
\item $z_{n^2}=1$ for all $n\ge0$ and $z_i=0$ for all $i\in\mathbb{Z}\setminus\{n^2\colon n\ge0\}$.
\end{itemize}
Let $Y=\{x\}\cup\{\sigma^i(y)\colon i\in\mathbb{Z}\}$ and let $X=Y\cup\{\sigma^i(z)\colon i\in\mathbb{Z}\}$. Let $f=\sigma|_X\colon X\to X$. It holds that
\begin{itemize}
\item $z\in C(\omega_f)$,
\item $\omega_f(z)=\Omega_f(z)=Y$,
\item $Y$ is not a minimal set for $f$,
\item $USC(\omega_f)=X\setminus Y$ and $LSC(\omega_f)=X$,
\item $Sh(f)=\emptyset$.
\end{itemize}
\end{ex}

\begin{ex}
\normalfont
Let $Y=\{y\in\mathbb{C}\colon|y|=1\}$ and let $X=Y\times[0,1]$. Let $f\colon X\to X$ be the map defined by $f(y,t)=y\cdot e^{2\pi it}$ for all $(y,t)\in X$. It holds that
\begin{itemize}
\item $USC(\omega_f)=\{(y,t)\in X\colon t\not\in\mathbb{Q}\}$ and $LSC(\omega_f)=X$,
\item $\omega_f((1,0))=(1,0)$ and $\Omega_f((1,0))=Y\times\{0\}$,
\item $Sh(f)=\emptyset$.
\end{itemize}
\end{ex}

\begin{ex}
\normalfont
Let $X=\mathbb{R}^2/\mathbb{Z}^2$ and let $f\colon X\to X$ be the map induced from
\[
A=
\begin{pmatrix}
2 & 1 \\
1 & 1 \\
\end{pmatrix}
.
\]
Let $S=\{x\in X\colon\omega_f(x)=X\}$ and note that $S$ is a dense $G_\delta$-subset of $X$ because $f$ is transitive. It holds that
\begin{itemize}
\item $X=Sh(f)$ because $f$ is an Anosov diffeomorphism,
\item $\Omega_f(x)=X$ for all $x\in X$,
\item $USC(\omega_f)=S$ and $LSC(\omega_f)=\emptyset$.
\end{itemize}
\end{ex}

\begin{ex}
\normalfont
Let $X=\{x\in\mathbb{C}\colon|x|=1\}$. Let $\alpha\in[0,1)\setminus\mathbb{Q}$ and let $f\colon X\to X$ be the map defined by $f(x)=x\cdot e^{2\pi i\alpha}$ for all $x\in X$. It holds that
\begin{itemize}
\item $X$ is connected,
\item $X=C(\omega_f)$ because $\omega_f(x)=X$ for all $x\in X$,
\item $Sh(f)=\emptyset$.
\end{itemize}
\end{ex}

\begin{ex}
\normalfont
Given a continuous map $f\colon X\to X$, let $\Omega(f)$ denote the {\em non-wandering set}  for $f$:
\[
\Omega(f)=\{x\in X\colon x\in\Omega_f(x)\}.
\]
Note that $\Omega(f)\subset CR(f)$. Let $X=\{x\in\mathbb{C}\colon|x|=1\}$ and let $f\colon X\to X$ be an orientation-preserving homeomorphism such that
\begin{itemize}
\item the rotation number of $f$ is irrational, 
\item $f$ is not transitive
\end{itemize}
(e.g.\:the Denjoy example). It holds that
\begin{itemize}
\item $X$ is connected,
\item $\Omega(f)$ is a Cantor set and so totally disconnected,
\item $\Omega(f)$ is a minimal set for $f$,
\item $X=C(\omega_f)$ because $\omega_f(x)=\Omega(f)$ for all $x\in X$,
\item $X=CR(f)$,
\item $Sh(f)=\emptyset$.
\end{itemize}
\end{ex}

\appendix

\section{}

The aim of this appendix is to prove the following theorem.

\begin{thm}
For a continuous map $f\colon X\to X$, if $X=USC(\omega_f)$, then $X=LSC(\omega_f)$.
\end{thm}

Let us first prove the following simple lemma.

\begin{lem}
Let $f\colon X\to X$ be a continuous map. For any $x\in X$ and $y\in\Omega_f(x)$, if
\[
\{x,y\}\subset USC(\omega_f),
\]
then $\omega_f(x)\cap\omega_f(y)\ne\emptyset$.
\end{lem}

\begin{proof}
Assume the contrary, i.e., $\omega_f(x)\cap\omega_f(y)=\emptyset$. Then, there are open subsets $U,V$ of $X$ such that $\omega_f(x)\subset U$, $\omega_f(y)\subset V$, and $U\cap V=\emptyset$. Since $\{x,y\}\subset USC(\omega_f)$, there is $\delta>0$ such that
\begin{itemize}
\item $d(x,z)\le\delta$ implies $\omega_f(z)\subset U$ for all $z\in X$, 
\item $d(y,w)\le\delta$ implies $\omega_f(w)\subset V$ for all $w\in X$.
\end{itemize}
Since $y\in\Omega_f(x)$, we have
\[
\max\{d(x,z),d(y,f^i(z))\}\le\delta
\]
for some $z\in X$ and $i\ge0$. It follows that $\omega_f(z)\subset U$ and
\[
\omega_f(z)=\omega_f(f^i(z))\subset V,
\]
which imply $\omega_f(z)\subset U\cap V$ and thus $U\cap V\ne\emptyset$, a contradiction. This completes the proof of the lemma.
\end{proof}

By using Lemma A.1, we prove Theorem A.1.

\begin{proof}[Proof of Theorem A.1]
Assume that $X=USC(\omega_f)$. For any $x\in X$, we take a minimal set $M$ for $f$ such that $M\subset\omega_f(x)$. Fix $p\in M$ and note that $\omega_f(p)=M$. Let $U$ be an open subset of $X$ with $M\subset U$. Since $p\in USC(\omega_f)$, there is $\delta>0$ such that $d(p,z)\le\delta$ implies $\omega_f(z)\subset U$ for all $z\in X$. Since $p\in\omega_f(x)$, we have $d(p,f^i(x))\le\delta$ for some $i\ge0$. It follows that
\[
\omega_f(x)=\omega_f(f^i(x))\subset U.
\]
Since $U$ is arbitrary, we obtain $\omega_f(x)\subset M$ and thus $\omega_f(x)=M$. For any $y\in\Omega_f(x)$, as $\{x,y\}\subset USC(\omega_f)$, by Lemma A.1, we have 
\[
M\cap\omega_f(y)=\omega_f(x)\cap\omega_f(y)\ne\emptyset
\]
and so
\[
\omega_f(x)=M\subset\omega_f(y).
\]
Because $y\in\Omega_f(x)$ is arbitrary, by Lemma 3.1, we obtain $x\in LSC(\omega_f)$. Since $x\in X$ is arbitrary, we conclude that $X=LSC(\omega_f)$, completing the proof of the theorem. 
\end{proof}

We shall recall the definition of $\omega$-chaos \cite{L}. Let $f\colon X\to X$ be a continuous map. A subset $S$ of $X$ is said to be an {\em $\omega$-scrambled set} for $f$ if for any $x,y\in S$ with $x\ne y$,
\begin{itemize}
\item $\omega_f(x)\setminus\omega_f(y)$ is an uncountable set,
\item $\omega_f(x)\cap\omega_f(y)\ne\emptyset$,
\item $\omega_f(x)$ is not contained in the set of periodic points for $f$.
\end{itemize}
We say that $f$ is {\em $\omega$-chaotic} if there is an uncountable $\omega$-scrambled set for $f$.

\begin{rem}
\normalfont
Given a continuous map $f\colon X\to X$, the above proof of Theorem A.1 shows that if $X=USC(\omega_f)$, then $\omega_f(x)$ is a minimal set for $f$ for all $x\in X$; therefore, $\omega_f(x)\cap\omega_f(y)=\emptyset$ or $\omega_f(x)=\omega_f(y)$ for all $x,y\in X$. It follows that if $X=USC(\omega_f)$, then there is no $\omega$-scrambled set for $f$ that is not a singleton, and thus $f$ is not $\omega$-chaotic.
\end{rem}

\end{document}